\documentclass[11pt,leqno]{amsart}

\usepackage[english]{babel}
\usepackage{a4wide}
\usepackage[utf8]{inputenc}

\usepackage{enumitem}
\setlist{leftmargin=5mm}

\usepackage{amsmath}
\usepackage{amssymb}
\usepackage{amsthm}
\usepackage{amsfonts}
\usepackage{amsmath}
\usepackage{latexsym}
\usepackage{euscript}
\usepackage{graphicx}
\usepackage[table]{xcolor}
\usepackage{listings}
\usepackage{url}
\usepackage{blindtext}
\usepackage{hyperref}
\usepackage{xcolor}
\usepackage{multicol}
\usepackage{tikz}
\usetikzlibrary{arrows,automata,positioning}

\usepackage{color}

\newtheorem{example}{Example}

\newtheorem{theorem}{Theorem}

\newtheorem{lemma}{Lemma}
\newtheorem{observation}{Observation}

\newtheorem*{definition}{Definition}

 \everymath{\displaystyle}

\def\L{{\mathcal L}}
\def\M{{\mathcal M}}

\def\A{{\EuScript A}}
\def\B{{\EuScript B}}
\def\C{{\EuScript C}}
\def\D{{\EuScript D}}

\def\P{{\EuScript P}}
\def\Q{{\EuScript Q}}
\def\R{{\EuScript R}}
\def\S{{\EuScript S}}

\def\Z{\mathbb{Z}}

\newcommand\wordlength[1]{|#1|}

\newcommand\word[3]{{#1}_{#2}\cdots{#1}_{#3}}

\newcommand\pair[2]{\langle#1, #2\rangle}

\newcommand\wordconcat[2]{#1\,#2}
\newcommand\preduction[1]{\overline{#1}}
\newcommand\pexpansion[1]{\widetilde{#1}}
\newcommand\prefix[2]{#1_{#2}}

\newcommand\suboperator[1]{\theta #1}
\newcommand\suboperatornext{\pair{[\prefix{A}{i-1} (a_i - 1) (s - 1)^{j - i}]^{q} \prefix{A}{n-qj}}{0}}
\newcommand\suboperatorletternext[1]{\pair{[\prefix{#1}{i-1} ({\lowercase{#1}}_i - 1) (s - 1)^{j - i}]^{q} \prefix{#1}{n-qj}}{0}}

\newcommand{\maximal}{maximal }

\makeatletter
\def\@settitle{\begin{center}%
  \baselineskip14\p@\relax
    \normalfont\LARGE%<- NEW
  \@title
  \end{center}%
}
\makeatother

\begin{document}
\let\MakeUppercase\relax % this disables uppercasing authors

\title[Lyndon pairs and 
the  lexicographically greatest perfect necklace]{\Large Lyndon pairs and 
the  lexicographically greatest perfect necklace}

\author[Verónica Becher and Tomás Tropea]{
\begin{tabular}{cc}
Ver\'onica Becher \hspace{2cm}& \hspace{2cm} Tomás Tropea\\
{\tt vbecher@dc.uba.ar} & \hspace{2cm}{\tt tomastropeaa@gmail.com}\\
\end{tabular}\\\\
{\small  Departamento  Computaci\'on, Facultad de Ciencias Exactas y Naturales, Universidad de Buenos Aires\\and ICC CONICET.
 Pabell\'on 0, Ciudad Universitaria, C1428EGA Buenos Aires, Argentina}
}

\date{\today}

%68R15 Combinatorics on words}
%05A05 Permutations words matrices
%11K16 normal numbers


\maketitle

\begin{abstract}
Fix a finite alphabet. A necklace is a circular word. For  positive integers $n$ and~$k$, a necklace is $(n,k)$-perfect if  all words of length $n$  occur $k$ times but at positions with  different congruence modulo $k$, for any convention of the starting position. We define the notion of a Lyndon pair and we use it  to construct the lexicographically greatest $(n,k)$-perfect necklace, for any $n$ and $k$ such that~$n$ divides~$k$ or $k$ divides~$n$.
Our construction generalizes  Fredricksen and Maiorana's construction of the lexicographically greatest de Bruijn sequence of order $n$, based on the concatenation of the Lyndon words whose length divide $n$. 
\end{abstract}

%\tableofcontents

\section{Introduction}

Let $\Sigma$ be a finite alphabet with at least two symbols.
A word on $\Sigma $ is a finite sequence of symbols, and a necklace is the equivalence class of a word under rotations.
Given two positive integers, $n$ and $k$, a necklace is $(n,k)$-perfect if  all words of length $n$ 
occur $k$ times but at positions with 
different congruence modulo $k$, for any convention of the starting position.
% Given a positive integer $n$, a necklace is de Bruijn of order $n$ if every word of  length $n$ occurs in it exactly once
The well known circular de Bruijn  sequences  of order $n$, see~\cite{deBruijn1946,Kor51,Kor52}, 
%that we call de Bruijn necklaces of order $n$, 
are exactly the $(n,k)$-perfect necklaces for $k=1$.
For example, $11100100$ is a $(2,2)$-perfect for $\Sigma=\{0,1\}$.
The $(n,k)$-perfect necklaces correspond to Hamiltonian cycles in the tensor product of the de Bruijn graph with a simple cycle of length $k$.

A thorough presentation of  perfect necklaces appears in~\cite{ABFY}. With the purpose of constructing normal numbers with very fast convergence to normality M. Levin in  \cite{Levin1999} gives two  constructions of perfect necklaces. 
One  based on arithmetic progressions with difference coprime with the alphabet size which  yields $(n,n)$-perfect necklaces. The other based on Pascal triangle  matrix which yields nested $(n,n)$-perfect necklaces when $n$ is a power of $2$.
In~\cite{BecherCarton2019} there is a method of constructing all nested $(n,n)$-perfect necklaces for the alphabet~$\{0,1\}$.

Assume the lexicographic  order on words.
A Lyndon word is a nonempty aperiodic word that is lexicographically greater than  all of its rotations.
For example, the Lyndon words over  alphabet $\{0,1\}$
sorted by length and then in decreasing lexicographical order within each length  are
\[
1,0, 10,  110, 100, 
1110, 1100, 1000, 
11110,11100,11010,11000,
  10100,    10000, \ldots 
\]

Lyndon words were introduced by Lyndon in the  1950s \cite{Lyndon,LyndonII}.
They  provide a nice factorization of the free monoid $\Sigma^*$:
 each word $w$ in $\Sigma^*$  
has a unique decomposition as a product 
$w = u_1 \ldots u_n$ of 
a non-increasing sequence of Lyndon words $u_1, \ldots u_n$
in  the lexicographic order. 
The problem to compute the  prime factorization of a given word 
has  a
solution in time linear to the length of the given word  \cite{Duval}, see also \cite{Knuth3}.

Fredricksen and Maiorana \cite{FM78} construct a de Bruijn sequence of order $n$ by concatenating all the Lyndon words whose length divides~$n$.
They first identify each necklace with the word that represents the  lexicographically maximal rotation. Order the necklaces according to lexicographical order of these words.
Fredericksen and Maiorana de Bruijn sequence of order $n$ is the concatenation, according to this order of the necklaces,  of the respective   Lyndon words having  length divisible by  $n$.
For example,   the binary words of length $4$  yield $6$ necklaces, and the words representing the  lexicographically maximal rotations,  in decreasing lexicographical order are:
\[
1111, \ \ 
1110, \ \ 
1100, \ \ 
1010, \ \ 
1000, \ \ 
0000.  
\]
The corresponding Lyndon words,  in the above order are:
\[
1, \ \ 1110, 1100,\ \ 10, \ \ 1000,\ \ 0.
\]
The length of each of them divides $n=4$, because it is $1$ or $2$ or $4$. Then,  none of  them is discarded, hence,   
Fredericksen and Maiorana's de Bruijn sequence of order $n$ sequence is: 
\[
1\ 1110\  1100\  10\  1000\ 0.
\]
% \textcolor{red}{
% La concatenacion se hace en estos pasos:
% \begin{itemize}
%     \item Busco de mayor a menor las palabras que son la maxima rotacion
%     \item Luego reduzco estas para obtener lyndon words
%     \item Finalmente las concateno en el orden que las recupere.
% \end{itemize}
% }
This construction, together with the efficient generation of Lyndon words,  provides a method for constructing  the lexicographically greatest  de Bruijn necklace of each order $n$
 in linear time and logarithmic space.

In this note we present the notion of Lyndon pairs and we use it to  generalize Fredricksen and Maiorana's algorithm  to construct the lexicographically greatest $(n,k)$-perfect necklaces,
for any $n$ and $k$ such that~$n$ divides~$k$ or $k$ divides~$n$.
Our presentation also 
%gives a clear elaboration of 
elaborates  the rather briefly presented ideas and proofs in \cite{FM78}.

\section{Lyndon pairs in lexicographical order}

\subsection{Lyndon pairs}

We assume a finite alphabet $\Sigma$ with cardinality $s$, with $s\geq 2$.
Without loss of generality we assume  $\Sigma=\{0, \ldots , s-1\}$.
We use lowercase letters $a, b, c$ possibly with subindices for alphabet symbols.
Words are finite sequences of symbols that we write  $a_1 a_2 \ldots a_n$ or with a capital letter  $A,B, C$.
We write $a^\ell$ to denote the word of length $\ell$ made just of~$a'$s.
and we write $A^\ell$ to denote the word made of $\ell$ copies of $A$.
The concatenation of two words $A$ and $B$ is written $AB$.
The length of a word $A$ is denoted with $|A|$.
The positions of a word $A$ are numbered  from $1$ to $|A|$.
We use $>$ to denote the decreasing lexicographic order on words and we write 
$A\geq B$  when $A > B$ or $A=B$.
%Many times wAnd we write $A<B$, A\leq B 

We use lowercase letters $h,\ldots, z$ to denote non-negative integers.
We write $k|n$ to say that~$k$ divides $n$.
We write $\Z_k$ for the set $\Z/k\Z$ of residues modulo~$k$.
We also use $<$ and $>$ for the natural orders on $\Z_k$ and $\mathbb{N}$ and, as usual,   $u\leq v$  when $u< v$ or $u =v$; and $v\geq u$ when $v>u $ or $v=u$.
When $u<v$ we may write $v>u$.

We work with pairs in $\Sigma^n\times \Z_k$ 
when $k|n$  or $n|k$.
This condition on $n$ and $k$ is  assumed   all along the sequel.
We refer to pairs $\pair{A}{u}$ with calligraphic letter $\A,\B,\C, \ldots$.
If $\A=\pair{A}{u}$ in  $\Sigma^n\times \Z_k$  we write $\A^p$ to denote  $\pair{A^p}{u}$ in $\Sigma^{pn}\times \Z_k$ .
We consider the following order  $\succ$ over  $\Sigma^n\times\Z_k$.

\begin{definition}[order $\succ$ on $\Sigma^n\times\Z_k$]\em
\[
\pair{A}{u} \succ \pair{B}{v}  \text{  exactly when  either }
  (A> B  \text{ and } u=v ) \text{ or } (k-1-u)>(k-1-v). 
  %(u<v).
\]
\end{definition}
The smallest the second component in $\Z_k$, the $\succ$-greater the pair. Among pairs with the same second component in $\Z_k$, the order $\succ$ is defined with the decreasing lexicographic order on $\Sigma^n$.
Thus, $\pair{(s-1)^n}{0}$ is the $\succ$-greatest  
in  $\Sigma^n\times\Z_k$
 and $\pair{0^n}{k-1}$ is the $\succ$-least.
As usual, we write
$\pair{A}{u}\succeq \pair{B}{v} $ exactly when $ \pair{A}{u}\succ\pair{B}{v}$ or $\pair{A}{u}=\pair{B}{v}$.

\begin{definition}[rotation of a pair]\em
Given a pair $\pair{a_1\ldots a_n}{u}$ in $\Sigma^n\times\Z_k$
its right rotation is the pair 
$\pair{a_n a_1\ldots a_{n-1}}{u-1}$ and its left rotation is  the pair $\pair{a_2\ldots a_na_1}{ u+1}$.
\end{definition}

For $s = 3$, $n=k = 5$
and  $\pair{13212}{4}$, 
 its right rotation is $\pair{21321}{3}$, and 
its left rotation  
is $\pair{32121}{0}$.
The rotation function induces a relation between pairs:
two pairs are related if successive rotations initially applied to the first yield the second.
This relation is clearly reflexive and transitive.
For pairs in   $\Sigma^n\times \Z_k$, 
when  $k|n$  or $n|k$,
the rotation has an inverse, given by successive rotations.
So the relation is also symmetric, hence,  an equivalence relation.

\begin{definition}[\maximal pair]\em
A necklace in $\Sigma^n\times\Z_k$
 is a set of pairs that are equivalent under rotations.
 For  each $u=0, \ldots k-1$, the following set is a necklace
\[
 \{\pair{a_{i+1}\ldots a_na_1\ldots a_{i}}{u+i}:  0\leq i < \max(n,k) \}.
 \]
In each necklace we are interested in the pair that is maximal in the order $\succ$. 
If a pair $\A$ in $\Sigma^n\times\Z_k$ is $\succeq$-maximal among its rotations we call it 
maximal.
\end{definition}
For example, for $n=4$ and $k=2$ 
the pair $\pair{1110}{0}$ is \maximal because it is $\succ$-greater than  its rotations 
$\pair{1101}{1}$,
  $\pair{1011}{0}$ and
    $\pair{0111}{1}$. 
The pair  $\pair{1010}{0}$ is \maximal because it is $\succ$-greater than  its rotation
$\pair{0101}{1}$.

\begin{observation}\label{obs:todos-maximales}
When $n|k$  
the  \maximal pairs  in 
$\Sigma^n\times\Z_k$
are the pairs $\langle A,0\rangle$ for $A\in \Sigma^n$.
\end{observation}

We can concatenate pairs in 
$\Sigma^n\times \Z_k$ 
 having the same second component:
the concatenation of  $\pair{A}{u}$ and  $\pair{B}{u}$ is $\pair{AB}{u}$.
Given $\A=\pair{A}{u}$ we write $\A^p$ to denote the pair $\pair{A^p}{u}$.

\begin{observation}\label{obs:copias}
Assume  $k|n$ or $n|k$.
Then, $\A \in \Sigma^n\times\Z_k$ is  \maximal exactly when,
 for any $p\geq 1$,
 $\A^p$
 is \maximal in $\Sigma^{pn}\times\Z_k$.
\end{observation}

\begin{proof}
($\implies$).
If $n|k$ then all pairs in $\Sigma^n\times \Z_k$ with second component $0$  are \maximal.
We prove it for  $k|n$, $k<n$.
Assume  $\A=\pair{A}{u}$ such that $\A^p=\pair{A^p}{u}$ is \maximal, but $\A$ is not maximal. 
Then $\A$ has a rotation with $j\geq 1$ such that 
$
\R = \pair{(\word{a}{j+1}{n}\word{a}{1}{j})}{u+j}$
and $\R\succ\A$.
Then either,
\[
\word{a}{j+1}{n}\word{a}{1}{j}> \word{a}{1}{n}, \text{ or }
(k-1-(u+j))>(k-1-u).
\]
But this implies that either
\[
(\word{a}{j+1}{n}\word{a}{1}{j})^p
>
(\word{a}{1}{n})^p, \text{ or } (k-1-(u+j))>(k-1-u).
\]
This, in turn, implies  $\R^p\succ \A^p$, which  contradicts that  $\A^p$ is maximal.

$(\impliedby)$.
Assume  $\A$ is \maximal but $\A^p$ is not. 
Then, since $\A$ is maximal $\A=\pair{A}{0}$ and $\A^p=\pair{A^p}{0}$.
Since $\A^p$ is not maximal it has a rotation such that  $\R\succ \A^p$. 
Then, necessarily,
\[
\R=\pair{R}{0}=\pair{\word{a}{1+j}{n}\word{a}{1}{j}A^{p-1}}{0}.
\]
Given that the second component of both $\R$ and $\A^p$ is $0$,  necessarily $R>  A^p$. But this contradicts that $\A=\pair{A}{0}$ was a \maximal rotation.
\end{proof}

\begin{observation} \label{obs:reduccion-maximo} 
If $k|n$ and  $\A\in \Sigma^n\times \Z_k$ is a \maximal pair then none of its rotations are $\succ$-greater than $\A$, but there may be  a rotation that is equal to $\A$. 
\end{observation}

For a word $A=a_1\ldots a_n$ we write  $A_i $ to denote its prefix of length $i$, that is, $a_1\ldots a_i$.

\begin{lemma}\label{lemma:es-maximal}
Let $n$ and $k$ be positive integers such that  $k|n$ or $n|k$. 
Let  
$\A = \pair{A}{0}$
 in $\Sigma^n\times \Z_k$ be \maximal and
different from $\pair{0^n}{0}$.
Suppose  $A=a_1\ldots a_n$, let $i$ be  such that  $a_i>0 $ and let 
\[
\B = \pair{\prefix{A}{i-1} (a_i - 1) (s-1)^{j-i}}{0}, \]
where 
if $n|k$ then $j=n$;
and  if $k|n$ then $j$ is the smallest such that $k|j$ and $i\leq j< n$.
Then, $\B\in\Sigma^j\times \Z_k$ is maximal.
%where  $j$ is the smallest multiple of $k$ greater than or equal to $i$
%That is,  $j = i + ((n - i) \mod k)$.
\end{lemma}
\begin{proof}
Let $\A$ be a \maximal pair .
By way of contradiction, assume 
$\B$=$\pair{\prefix{A}{i-1} (a_i - 1) (s-1)^{j-i}}{0}$ is not a \maximal pair.
Then there is some~$\ell$ multiple of $k$ such that :
    \begin{center}
    $\pair{\word{a}{\ell + 1}{i-1} (a_i - 1) (s-1)^{j-i} \prefix{A}{\ell}}{0 +\ell} \succ \pair{\prefix{A}{i-1} (a_i - 1) (s-1)^{j-i}}{0}$
    \end{center}
Necessarily
\[   \word{a}{\ell + 1}{i-1} \geq 
   \word{a}{1}{(i - 1) - (\ell + 1) + 1}.
\]
Since  $\A$ is a \maximal pair,
\[
 \word{a}{1}{(i - 1) - (\ell + 1) + 1} \geq 
 \word{a}{\ell + 1}{i-1}.
\]
Therefore, $a_{\ell+1} = a_1$, $a_{\ell+2} = a_2$, $\ldots $, $a_{i-1} = a_{(i - 1) - (\ell+1) + 1}$. This implies, $a_i - 1 \geq a_{(i - 1) - (\ell+1) + 2}$. 
Consequently, 
\[
   \word{a}{\ell + 1}{i-1} a_i > 
   \word{a}{1}{(i - 1) - (\ell + 1) + 2},
\]
 but this contradicts that $\A$ is a \maximal  pair.
\end{proof}

For example, let $s = 7$, $n = 6$, $k = 3$ and   $\A=\pair{456123}{0}$. Since $\A$ is maximal and  the symbol in position $i=3$ is  $6$, letting $j=3$ (a multiple of $k$ grater than or equal to $i$) we obtain 
in  $\Sigma^j\times \Z_k$ the maximal pair $\B=\pair{455}{0}$.
No filling with $s-1=6$ was required.
Since  the symbol  in position $i=1$ is $4$, letting $j=3$ we obtain the maximal pair $\B=\pair{366}{0}$ which is obtained by concatenating $s-1=6$ $(j-i)=2$ times.

\begin{definition}[reduction]
Let $n$ and $k$ be positive integers such that $k|n$.
The reduction of a word  $A= \word{a}{1}{n}$ is the word
 $\preduction{A} 
 =\word{a}{1}{p}$ where  
 $ (\word{a}{1}{p})^{n / p}=\word{a}{1}{n}$ and 
 $p$ is the smallest  such that 
 $k \mid p$, $p \mid n$.
 The reduction of a pair  $\A = \pair{A}{u}$ in $\Sigma^n\times\Z_k$, 
is  the  pair 
$\preduction{\A}=\pair{\preduction{A}}{u}$ in $\Sigma^p\times\Z_k$.    \end{definition}

When $k|n$, the reduction always exists because  one can take  $p=n$.   
%Notice that if a pair  in $\Sigma^n\times \Z_k$ is non reduced with period $p$, such that  $k|p$ and $p|n$  then it has $n/p$ equal rotations. 
%However, 
Notice that all the rotations of a reduced pair are pairwise different.
For example, for  $s = 8$, $n = 8$ and  $k = 2$, 
 $\preduction{\pair{10101010}{0}}  =  \pair{10}{0}$;
 $\preduction{\pair{01230123}{0}}  =  \pair{0123}{0}$;
$\preduction{\pair{01234567}{0}}  =  \pair{01234567}{0}$.

\begin{definition}[expansion]
Let $n$ and $k$ be positive integers such that $n|k$.
The expansion of a word  $A= \word{a}{1}{n}$ is the word 
 $\pexpansion{A}=(\word{a}{1}{n})^{k/n}$.
The expansion of a pair 
$\A = \pair{A}{u}$   in $\Sigma^n\times \Z_k$ is the pair $\pexpansion{\A}=
\pair{\pexpansion{A}}{u}$
in 
$\Sigma^k\times \Z_k$.
\end{definition}

When   $n|k$ the expansion always exists.
For example for  $s = 3$, $n = 2$ and $k = 8$, 
 $\pexpansion{\pair{12}{0}} =
  \pair{12121212}{0}$.
  Notice that when $k=n$ then $\pexpansion{A}=A$.
 
\begin{observation} \label{observacion:prefijo-expansion} 

If  $n|k$ then for  every  pair $\A=\pair{A}{u}\in \Sigma^n\times \Z_k$, $A$ is a prefix of $\pexpansion{A}$.
\end{observation}

\begin{definition}[Lyndon pair]
Let $n$ and $k$ be positive integers. 
When   $k|n$, the Lyndon pairs are  the reductions of the \maximal pairs in  $\Sigma^{n}\times{\Z_k}$. 
When  $n|k$ with $n<k$, the  Lyndon pairs are the expansions  of the \maximal pairs $\pair{A}{0}\in \Sigma^{n}\times\Z_k$.
\end{definition}
Thus, when   $k|n$, the Lyndon pairs    are elements in $\bigcup_{p: k|p|n}\Sigma^{p}\times\Z_k$.
 But when $n|k$ the  Lyndon pairs    are elements in $\Sigma^{k}\times\Z_k$.

\begin{observation} \label{observacion:reduccion-maximo} 
Every Lyndon pair is 
strictly  $\succ$-greater
than each of its rotations.
\end{observation}

\subsection{The operator $\suboperator{}$}

We define the operator $\suboperator{}$ that given a pair in $\Sigma^n\times \Z_k$
 different from  $\pair{0^n}{0}$ but 
  with second component $0$,
  it defines another pair  in $\Sigma^n\times \Z_k$
with second component~$0$.

\begin{definition}[operator $\suboperator{}$]
Let $n$ and $k$ be positive integers such that  $k|n$ or $n|k$.
    For $\A=\pair{A}{0}=\pair{\word{a}{1}{n}}{0}$ in $\Sigma^n\times \Z_k$
such that  $a_i > a_{i+1} = \ldots  = a_n = 0$, 
we define the operator $\theta:\Sigma^n\times\Z_k\to \Sigma^n\times\Z_k$,
  \[ \suboperator{\pair{A}{0}} = \suboperatornext,\]
   where,

if  $n|k$  then $j=n$ and $q=1$. So,   $ \suboperator{\pair{A}{0}} = 
  \langle [A_{i-1}(a_i-1)(s-1)^{n-i},0\rangle$;

if $k|n$ and $n<k$
then $j$ is the smallest integer such that $k|j$ and $j\geq i$,
and  $q$ is the greatest integer such that $q\leq n/j$. % $0 \leq n - qj < j$.
Thus, in either case,  
$j = i + ((n - i) \mod k)$.
\end{definition}

The operator $\suboperator{}$  %is a function from $\Sigma^n\times \Z_k$to $\Sigma^n\times \Z_k$. It
is applicable on any pair with second component $0$, except for 
$\pair{0^n}{0}$.
For example, 
for $s=2$, $n=6$ and $k = 2$,   
$\suboperator{\pair{010000}{0}}=\pair{000000}{0}$;
$\suboperator{\pair{011000}{0}}=\pair{010101}{0}$;
$\suboperator{\pair{011101}{0}}=\pair{011100}{0}$.

\begin{definition}[value $T$]
Let  $T$ be the integer such that $\suboperator{}^T\pair{(s-1)^n}{0}=\pair{(0^n}{0})$.
\end{definition}

\begin{lemma}\label{lemma:decreciente}
The list $(\suboperator{}^i \pair{(s-1)^n}{0})_{0\leq i\leq T}$ 
is strictly decreasing in $\succ$.
\end{lemma}

\begin{proof}
Let's see that for every pair  $\A=\pair{A}{0}$,
$           \A \succ \suboperator{\A}$.
Using the definition
of $\suboperator{}$,
    \[
        \pair{A}{0} \succ \suboperatornext.
    \]
Since both pairs have second component $0$, there is some  $i$ such that
    \[
        \word{a}{1}{i} 
        = \prefix{A}{i-1} a_i 
        > \prefix{A}{i-1}(a_i - 1).
    \]
We conclude that $(\suboperator{}^i \pair{(s-1)^n}{0})_{0\leq i\leq T}$ is strictly decreasing in  $\succ$.
\end{proof}

\begin{observation}\label{obs:bijection}
When $n|k$ the operator~$\theta$   yields a bijection between \maximal pairs.
 \end{observation}
\begin{proof}
Every  pair of the form $\pair{A}{0}$ in $\Sigma^n\times \Z_k$  is \maximal because,  when $n$ divides  $k$, there is just this unique rotation with second component $0$. 
The definition of  $\theta$  ensures that  $\theta$ goes through all the pairs of the form $\pair{A}{0}$
in $\succ$-decreasing order.
The operator $\theta$ can be used forward for every pair   $\pair{A}{0}$  except for $\langle 0^n,0\rangle$, and it can be used  backwards for every pair   $\pair{A}{0}$ except for  $\langle 0^n,0\rangle$. Thus, except for the extremes, we can obtain the successor and the predecessor of a maximal pair in the order  $\succ$.
\end{proof}

When $k|n$ and $k<n$  the operator $\theta$ 
 is not 
injective nor surjective over pairs with second component~$0$.
For example, for $s = 2$, $n = 4$ and $k = 2$, we see $\theta$ is not injective because 
$\suboperator{\pair{0100}{0}} = \pair{0000}{0}$
and also
$\suboperator{\pair{0001}{0}} = \pair{0000}{0}$.
To see that $\theta$ is not surjective  observe that  the pair  $\B=\pair{1011}{0}$ is not in the image of $\suboperator{}$,
because there is no pair $\A$ such that  $\suboperator{\A} = \B$.

It is possible to construct  the reverse of list $(\suboperator{}^i \pair{(s-1)^n}{0})_{0\leq i\leq T}$.
In case $n|k$ the operator~$\theta$ defines a bijection between \maximal pairs.
In case $k|n$ and $k<n$, 
$\suboperator{}$ is not injective, 
there are pairs that have more than one preimage by  $\suboperator{}$.
However,  except for $\pair{(s-1)^n}{0}$ every element has one predecessor 
in the list $(\suboperator{}^i(\pair{s-1)^n}{0})_{0\leq i\leq T}$,
which is just one of the possible preimages by~$\suboperator{}$. 

\begin{lemma}\label{lemma:preimage}
Let $n$ and $k$ positive integers such that  $k|n$. 
Every element $\A=\pair{A}{0}$ in \linebreak
$(\suboperator{}^i \pair{(s-1)^n}{0})_{0\leq i\leq T}$, except 
$\pair{(s-1)^n}{0}$, 
has a  predecessor  which is
the preimage of $\A$ by $\suboperator{}$  given by
\[
\pair{\prefix{A}{u - 1} (a_{u} + 1)) 0^{n - u}}{0}
\]
\text{where } 
\[
A=\left(
\prefix{A}{u} (s-1)^{r - u}
\right)^{w} \prefix{A}{v} %= (\prefix{A}{u} (s-1)^{r-u})^{w} \prefix{A}{v}
\]
$a_u < (s-1)$, and
$r$ is the smallest multiple of  $k$
such that $v < r$ and $r-k \leq u \leq r$.
\end{lemma}

\begin{proof}
First notice that this factorization always exists
\[
A=(\prefix{A}{r})^{w} \prefix{A}{v} = (\prefix{A}{u} (s-1)^{r-u})^{w} \prefix{A}{v}.
\]
If $w=1$ the  $r = n$, $v = 0$ and $A = \prefix{A}{n} = \prefix{A}{r} = \prefix{A}{u} (s-1)^{r-u}$.
Let $\B=\pair{B}{0}$
be the pair obtained by undoing the transformation done by the $\theta$ operator, knowing that  $\A$ and $\B$ are in  $(\suboperator{}^i \pair{(s-1)^n}{0})_{0\leq i\leq T}$,
\[
 \suboperator{\B} = \suboperatorletternext{B} = \A
\]
where  $i$ satisfies  $b_{i+1} = \ldots  b_{n} = 0$, and $i\leq j$ and $k|j$.
Thus,
\[
[B_{i-1}(b_i-1)(s-1)^{j-i}]^qB_{n-qj}=
(\prefix{A}{u} (s-1)^{r-u})^{w} \prefix{A}{v}.
\]
The word  $B$ is determined by
$r = j,
u = i,
w=q,
v = n - qj$ and $A_{i-1}(a_i+1) =B_i$.
%
% \begin{eqnarray*}
% r &=& j
% \\
% u &=& i
% \\
% w &=& q
% \\
% v &=& n - qj
% \\
% A_{i-1}(a_i+1) &=&B_i.
% \end{eqnarray*}
%
Then,
\begin{eqnarray*}
\pair{\prefix{B}{i - 1} b_{i} 0^{n - j}}{0} = \pair{\prefix{A}{u - 1} (a_{u} + 1)) 0^{n - u}}{0},
\\
\suboperatorletternext{B} = [\prefix{A}{u - 1} a_u (s-1)^{r-u}]^{w} \prefix{A}{v}.
\end{eqnarray*}\end{proof}

\begin{lemma}\label{lemma:1}
If $\A \succ \B \succ \C$, with $\C = \suboperator{\A}$, then 
 $\B$  is not \maximal.
\end{lemma}

\begin{proof}
Let $\A=\pair{A}{0}$, $\B=\pair{B}{0}$, $\C=\suboperator{\A} = \pair{C}{0}$
and  indices  $i,j$  such that
\begin{eqnarray*}
A &= & \word{a}{1}{i} 0^{n-i}, \text{with $a_i > 0$},
\\
B &= & \word{b}{1}{n},
\\
C &= & \word{c}{1}{n} = [\word{a}{1}{i-1} (a_i - 1) (s - 1)^{j - i}]^{q} \word{a}{1}{n-qj}.
\end{eqnarray*}
Assume  $\A \succ \B \succ \C$ and, by way of contradiction  suppose   $\B$ is a \maximal pair.
Since $\A \succ \B$,
\[
b_1 = a_1,\quad
b_2 = a_2,\quad 
\ldots,\quad 
b_{i-1} = a_{i-1},
\]
and it should be 
$a_i > b_i$.
Since 
$\B \succ \C$
then  $b_i \geq c_i = a_i - 1$; hence,  $b_i = a_i - 1$.
Then, 
we have
$\prefix{B}{i} = \prefix{C}{i}$. 
And for 
$\ell=1, \ldots j-i$ we have
$b_{i + \ell} = s - 1$.
This is because    $\B \succ \C$ and $\word{c}{i+1}{i+\ell} = (s - 1)^{j - i}$, 
then
$\word{b}{i+1}{i+\ell} =  (s - 1)^{j - i}$,
given that  $s-1$ is  the lexicographically greatest symbol.

We now show that the other symbols in  $\B$ and $\C$ also coincide. 
Since $\B$ is a \maximal pair,  $B \geq \word{b}{j+1}{n} \word{b}{1}{j}$ and we know that  $a_1 = b_1 \geq b_{j+1}$. 
Since $\B \succ \C$ we have $b_{j+1} \geq c_{j+1} = a_1$, hence $b_{j+1} = a_1$.
Repeating this argument we obtain for
$m=1,\ldots ,q$ and $p=1,\ldots ,i-1$,

     $b_{mj+p} = a_p$,  for  $mj+p \leq n$,\nopagebreak
    
    $b_{mj+i} = a_i - 1$  and \nopagebreak
    
    $b_{mj + i + \ell} = s - 1$ for  $\ell=1, \ldots ,  j - i$.
\\
This would imply
$\B = \C = \suboperator{\A}$.
%We conclude that 
Thus, there is no \maximal pair $\B$  such that  $\A \succ \B \succ \C$.
\end{proof}

\section{Statement of Theorem~\ref{thm:perfectp}}

\begin{definition}[list $\M$ of  \maximal pairs of $\Sigma^n\times \Z_k$]
Let $n$ and $k$ be positive integers such that $n|k$ or $k|n$.
Define $\M$ by removing from 
$(\theta^i \pair{(s-1)^n}{0})_{0\leq i\leq T }$
the pairs that are not maximal.
In case $n|k$ all the elements are maximal, none is removed.
In case $k|n$ with $k<n$, for each pair $\A$ in $(\theta^i \pair{(s-1)^n}{0})_{0\leq i\leq T }$, except  $\A=\pair{0^n}{ 0}$, 
remove 
$\suboperator{}\A$, $\suboperator{}^2\A, \ldots \suboperator^{h-1}\A$ 
where 
 $h$ is the   least such that
 $\suboperator{}^h(\A)$ is \maximal.  
 Let $M$ be the number of elements of the list~$\M$.
 \end{definition}

\begin{definition}[list $\L$ of Lyndon pairs]
Let $n$ and $k$ be  positive integers such that $k|n$ or $n|k$.
We define the list $\L$ as the list of  Lyndon pairs of the elements in the list $\M$.
Since $\M$ has $M$ elements, $\L$ has $M$ elements as well.
 \end{definition}

Recall that a necklace is a circular word and a  necklace is $(n,k)$-perfect if  all words of length~$n$ occur~$k$ times but at positions with 
different congruence modulo~$k$, for any convention of the starting position.

\begin{theorem}
\label{thm:perfectp}
Let $n$ and $k$ be  positive integers such that $k|n$ or $n|k$.
 Let $X$ be the concatenation of all the words in the  Lyndon pairs
 in the order given by $\L$.
Then, $X$ is  the lexicographically greatest
$(n,k)$-perfect necklace.
\end{theorem}

Here is an example for   $s = 2$, $n = 6$ and $k = 2$.
The lexicographically greatest $(n,k)$-perfect necklace is obtained by concatenating the following words (the symbol $\mid$ is   for ease of reading):

\begin{center}\small
        $11 \mid 111110 \mid 111101 \mid 111100 \mid 111010 \mid 111001 \mid 111000 \mid 110110 \mid 110101 \mid 110100 \mid 110010 \mid 110001 \mid 110000 \mid 10 \mid 101001 \mid 101000 \mid 100101 \mid 100100 \mid 100001 \mid 100000 \mid 01 \mid 010100 \mid 010000 \mid 00$.
\end{center}

\section{Proof of Theorem~\ref{thm:perfectp}}

We divide the proof of Theorem~\ref{thm:perfectp} in three parts:

in Section~\ref{sec:1} we show  that length of $X$ is $s^nk$, 

in Section~\ref{sec:2} we prove that $X$ is $(n,k)$-perfect,  and 

in Section~\ref{sec:3} we show that $X$ is the lexicographically greatest $(n,k)$-perfect necklace.
\\
We  start with two  lemmas about the list $\M$ of maximal pairs.

\begin{lemma}\label{lemma:lista}
The list $\M$ 
 starts with the $\succ$-greatest pair $\pair{(s-1)^n}{0}$, ends with the $\succ$-smallest pair $\pair{0^n}{0}$, and contains all \maximal pairs in strictly decreasing $\succ$-order.
\end{lemma}

\begin{proof}[Proof of Lemma~\ref{lemma:lista}] 
{\em Case $k \mid n, k<n$}.
\begin{itemize}

\item By definition, the list  $\M$ starts with  $\pair{(s-1)^n}{0}$.
    
\item The \maximal pairs are in $\succ$-decreasing order:
    This is because the list $\M$ is constructed by successive applications of the  operator $\suboperator{}$ and by  Lemma~\ref{lemma:decreciente}  the list is strictly  $\succ$-decreasing.
    
\item
No  \maximal pair is missing:
Suppose $\A$ is in $\M$ and let $ h$ be the least such that $\suboperator{}^h(\A)$ is \maximal. To argue by contradiction, suppose  there is 
    a \maximal pair $\B$ such that $\A\succ \B\succ \suboperator{}^h(\A)$. 
Then,  there is  $i$,  $0\leq i< h$, such that
\[
\A \succ \suboperator{\A} \succ \ldots \succ \suboperator{}^{i}{\A} \succ \B \succ \suboperator{}^{i+1}{\A} \succ\ldots \succ \suboperator{}^{h-1}{\A} \succ \suboperator{}^{h}{\A}.
\]\nopagebreak\samepage
Thus, $\B$ occurs in between  some  $\D$ and  $\suboperator{\D}$.
By Lemma~\ref{lemma:1}  this is impossible.
    
\item The list $\M$ ends with  $\pair{0^n}{0}$:
There is no $\A$ such that  $\pair{0^n}{0}\succ \A$, 
and the list $\L$ is strictly  $\succ$-decreasing. 
By Lemma~\ref{lemma:decreciente}, the operator 
 $\suboperator{} $ applies to any pair  except   $\pair{0^n}{0}$.
\end{itemize}

{\em Case $n \mid k$, $n\leq k$.} It is the same proof as in the previous case, but simpler because $\suboperator{}$ yields exactly all the  \maximal pairs in $\succ$-decreasing order.
\end{proof}

\begin{example}
 Let $s = 2$, $n = 6$ and $k = 2$. 
The list $\M$ of \maximal pairs is:

 \noindent
      $\pair{111111}{0},
        \pair{111110}{0},
        \pair{111101}{0},
        \pair{111100}{0},
        \pair{111010}{0},
        \pair{111001}{0},
       \pair{111000}{0},
       \pair{110110}{0},
     \\  \pair{110101}{0},
        \pair{110100}{0},
        \pair{110010}{0},
        \pair{110001}{0},
        \pair{110000}{0},
        \pair{101010}{0},
        \pair{101001}{0},
      \pair{101000}{0},
  \\     \pair{100101}{0},
        \pair{100100}{0},
        \pair{100001}{0},
        \pair{100000}{0},
        \pair{010101}{0},
       \pair{010100}{0},
        \pair{010000}{0},
        \pair{000000}{0}$.
\end{example}

The next is the key lemma.

\begin{lemma}\label{lemma:vecinos} 
Assume $k|n$. If $\A=\langle A,0\rangle$ is followed by  $\B=\langle B,0\rangle$ in the list of \maximal pairs $\M$ then $A$ is a prefix of $\wordconcat{\preduction{A}}{\preduction{B}}$.
\end{lemma}

\begin{proof}
  We can write  $\A$ 
  as 
  \[\preduction{\A}^{q} = \pair{(\prefix{A}{i}0^{p-i})^{q}}{0},
  \]
  where $q = n / p$ and  $a_i > 0$. 
  If  $q = 1$ then $\preduction{\A} = \A=\pair{A}{0}$ and  
  $\wordconcat{\preduction{\A}}{\preduction{\B}} = \pair{A}{0}\pair{\preduction{B}}{0}=\pair{A\preduction{B}}{0}$.
   Otherwise,  $q > 1$  and, since 
      $\B=\suboperator{\A}^h$ for the smallest $h$ such that it is maximal, the shape of~$\B$ is
   \[
   \B = \pair{\preduction{A}^{q-1} \prefix{A}{i-1}(a_i - 1)(s - 1)^{j-i}C}{0}
   \]
   for some word  $C$ and for $j$  the smallest such that  $i \leq j \leq p$ and $k \mid j$. 
   Since  $\B=\pair{B}{0}$ and 
 $B$ starts with $\preduction{A}^{q-1} \prefix{A}{i-1}(a_i - 1)$
 we have 
 $\pair{\preduction{B}}{0}=\pair{B}{0}$, hence  $\preduction{\B}=\B$.  Then,
        %TODO: MEJORAR JUSTIFICACION DE ARRIBA DE POR QUE EL B COLLAR TIENE ESA FORMA (QUIZAS SI NECESITAMOS ESE LEMA DE PEGAR 0s AL FINAL)
        % Basicamente argumentamos que seguro B tiene el A reducido q-1 veces, y dps tendria un bloque de 0 en el peor caso pero seria collar tambien.
            \begin{align*}
                \wordconcat{\preduction{\A}}{
                \preduction{\B}} = \pair{\preduction{A}}{0} \pair{\preduction{B}}{0} =
                \pair{\preduction{A}}{0}\pair{B}{0}=
                \pair{\preduction{A}B}{0}=
                \pair{\wordconcat{\preduction{A}}{
                \preduction{A}^{q-1}} \prefix{A}{i-1}(a_i - 1)(s - 1)^{j-i}C}{0}.
            \end{align*}
    In both cases $A$ is prefix of $\wordconcat{\preduction{A}}{\preduction{B}}$.
\end{proof}

\subsection{Proof that necklace $X$ has length $s^n k$} \label{sec:1}
\subsubsection{\bf When  $k|n, k< n$}
The necklace $X$ is the concatenation of all the  words 
of the reduced \maximal pairs exactly once.
Thus, the length of  $X$ is 
\[
|X|=\sum_{\preduction{\A}\text{ is Lyndon }} |\preduction{\A}|.
\]
Each $\preduction{\A}$ is an element of $\L$, which consists of the reductions  of the pairs in the list $\M$, which by Lemma~\ref{lemma:lista} contains all the maximal pairs.
 The length of $\preduction{A}=\pair{\preduction{A}}{0}$ is the length of the word $\preduction{A}$, which is  the number of different   rotations of $\A$, see Observation~\ref{observacion:reduccion-maximo}.
If $\preduction{\A}=\pair{\preduction{A}}{0}=\pair{\word{a}{1}{p}}{0}$,
then $\A$ has $p$ different rotations:
$\pair{\word{a}{1}{p}}{0}$,  $\pair{\word{a}{2}{n}a_1}{1}$,\ldots, $\pair{a_p\word{a}{1}{n-1}}{p-1}$. 
So, each of the $p$ positions in $\preduction{A}$  is the start of a different rotation of $\A$.
By Lemma~\ref{lemma:vecinos} 
if the successor of $\A=\pair{A}{0}$  in $\M$  is 
 $\B=\pair{B}{0}$ then  in $X$ we have  $\wordconcat{\preduction{A}}{\preduction{B}}$,
and  $A$ is a prefix of $\wordconcat{\preduction{A}}{\preduction{B}}$.
Let $X=a_0\ldots a_{\ell-1}$. We  argue that each 
position $i$, for $i=0, 1, \ldots, \ell-1$, is the start 
of one of the different  pairs $\pair{\word{a}{i}{i+n-1}}{i\mod k}$ in $\Sigma^n\times \Z_k$.
We conclude that $X$ has exactly one position  for each  of the different pairs  in $\Sigma^n\times \Z_k$. Thus, the length  of $X$ is  $s^nk$.

\subsubsection{\bf When  $n|k$, $n\leq k$}
The necklace $X$ is the concatenation of all the  words of the expanded \maximal pairs.
Thus, length of  $X$ is 
\[
|X|=\sum_{\pexpansion{\A}\text{ is Lyndon }} |\widetilde{\A}|.
\]
 The length of a Lyndon pair $\pexpansion{A}=\pair{\pexpansion{A}}{0}$  is  the length of the expansion of a maximal pair $\pair{A}{0}$ in $\Sigma^n\times \Z_k$, which is exactly $k$.
 % and it is also the number of different    rotations of $\A$.
There are $s^n$ many Lyndon pairs because, given that $n|k$,  every $\pair{A}{0}$ in $\Sigma^n\times \Z_k$ is maximal. 
  Since $X$ is the concatenation of all the Lyndon pairs, 
the length  of $X$ is  $s^nk$.

\subsection{Proof that necklace $X$ is $(n,k)$-perfect}\label{sec:2}

We need to show that each word of length~$n$ occurs exactly~$k$ times, at positions with different congruence modulo~$k$.
In this proof we number the positions of $X$ starting at $0$; this is convenient for the presentation because the positions with congruence $0$ are multiple of $k$.
We say that  we find a  pair $\pair{A}{u}$  in $X=a_0 a_1 \ldots a_{s^nk-1}$ 
when there is a position $p$ 
 such that $0\leq p<s^nk$, 
$p \equiv u \pmod{k}$,
 and
$a_p\ldots a_{p+|A|-1} = A$.
To prove that $X$ is a $(n,k)$-perfect necklace we need to find 
all the rotations of all the maximal pairs in $\Sigma^n\times\Z_k$  in the necklace $X$.

\subsubsection{\bf When  $k|n, k< n$}
Each maximal pair $\A=\pair{A}{0}$ has exactly $|\preduction{A}|$ many rotations.
%is always a multiple of $k$).
%We now consider each pair $\A=\pair{A}{0}$.
%\medskip
%as $\preduction{A}$ is used in the  construction of $X$

\paragraph{\em Case $A = (s-1)^n$.}
Since $k|n$,
$\preduction{\A} $=$\pair{\preduction{A}}{0} =\pair{(s-1)^k}{0}$, application of $\suboperator{}$ on $\A$ yields the \maximal pair
\[
 \suboperator{\A} = \B = \pair{B}{0} = \pair{(s-1)^{n-1} (s-2)}{0}.
\]
Since $\preduction{\A}$ and $\preduction{\B}$ are consecutive Lyndon pairs in $\L$, the construction of $X$ puts  $\preduction{A}$   followed by~$\preduction{B}$,
\[
\wordconcat{\preduction{\A}}{\preduction{\B}} = \pair{(s-1)^k}{0} \pair{(s-1)^{n-1} (s-2)}{0} = \pair{(s-1)^{n+k-1} (s-2)}{0}.
\]
This yields all rotations of  $\A$: $\pair{(s-1)^n}{0}$, \ldots ,  \mbox{$\pair{(s-1)^n}{k-1}$}.

\paragraph{\em Case $A = 0^n$.}
Since $k|n$, $\preduction{\A} =\pair{\preduction{A}}{0}=\pair{0^k}{0}$.
Consider the \maximal pairs for $i$=$0,1,\ldots, k-1$,
\[
\B = \pair{B}{0}=\pair{0^i 1 0^{k-1-i} 0^{n-k}}{0}.
\]
By Lemma~\ref{lemma:es-maximal}, for each of these $\B$ the next  \maximal pair in $\M$ is
\[
\suboperator{\B} = \C = \pair{C}{0}= \pair{(0^{i+1} (s-1)^{k-1-i})^{n/k}}{0}.
\]
Since $\preduction{\C} = \pair{0^{i+1} (s-1)^{k-1-i}}{0}$, 
it  is clear that  $\C$ is \maximal because all the rotations of $\C$ that have second component $0$ are identical to $\C$.
Since $\preduction{\B}$ and $\preduction{\C}$ are consecutive Lyndon pairs in $\L$, the construction of $X$ puts  $\preduction{B}$   followed by $\preduction{C}$. 
Notice that for each $i=0,1,\ldots, k-1$,
\[
\wordconcat{\preduction{\B}}{\preduction{\C}} = \pair{0^i 1 0^{k-1-i} 0^{n-k} 0^{i+1} (s-1)^{k-i-1}}{0} = \pair{0^i 1 0^n (s-1)^{k-i-1}}{0}
\]
gives rise to the pair $\pair{0^n}{i+1}$. 
Thus, we have all  the rotations   $\A = \pair{0^n}{0}$, which are   $\pair{0^n}{0}$, $\ldots $,  $\pair{0^n}{k-1}$.

\paragraph{\em Case $(s-1)^n > A >0^n$}
Every \maximal pair $\A = \pair{A}{0}$ different from  $\pair{0^n}{0}$ has the form \[ 
\A=\pair{(\prefix{A}{p})^{q+1}}{0},
\]
where  $\prefix{A}{p}$  is $\preduction{A}$ with  $p$ the smallest integer such that $A = (\prefix{A}{p})^{q+1}$ and  $q + 1 = (n / p)$.

{\em Subcase   $q>0$}. The pair $\A$, different from  $\pair{0^n}{0}$, always has a successor  \maximal pair  in the list $\M$:
\[
 \B = \suboperator{\A} = \pair{(\prefix{A}{p})^q \prefix{A}{i-1}(a_i - 1)(s-1)^{j-i}\word{b}{i+1}{p}}{0},
\]
where $i$ is the largest with $a_i>0$, $j$ is the smallest multiple of  $k$ with  $j \geq i$. Notice that  $\prefix{A}{p} \neq 0^p$. 
Since $\preduction{\B} = \B$, 
\[
\wordconcat{\preduction{\A}}{\preduction{\B}} =  \pair{(A_p)^{q+1}\prefix{A}{i-1}(a_i - 1)(s-1)^{j-i}\word{b}{i+1}{p}}{0},
\]
and it also yields the first $i$ left rotations of $\A$, which are 
$\pair{\word{a}{r+1}{n}\prefix{A}{r}}{r}$, with $0 \leq r < i$.
It remains to identify in the constructed necklace  $X$ the  $p-i$ right rotations of 
  $\A = \pair{(A_p)^{q+1}}{0}$.
These  are of the form 
 \[
\pair{a_{r}\ldots a_p A_p^{q}A_{p-i-r} }{-r},
\]
for $1 \leq r  \leq p-i$.
By Lemma~\ref{lemma:preimage} we can consider 
%To see this consider $\Q$ 
the predecessor $\Q$ of $\A$ by $\theta$,
\[
\suboperator{\Q}= \A= \pair{(\prefix{A}{p})^{q+1}}{0}  
\]
where
\[
\Q  = \pair{\prefix{A}{p-1}(a_p + 1)0^{pq}}{0}.
\]
To see that $\Q$ is a \maximal pair consider first $\R = \pair{\prefix{A}{p}}{0}$ which,  by Observation~\ref{obs:copias},   is  maximal.
We now argue that
$\prefix{A}{p}$
has no proper suffix 
$ \word{a}{i}{p}$ that coincides with
$\prefix{A}{p-i+1}$. 
If there were such a prefix we could construct the rotation of $\R$ given by the pair $\S = \pair{\word{a}{i}{p} \prefix{A}{i-1}}{0}$ 
and one of the following would be true:

\begin{itemize}
    \item $\S = \R$: but this is impossible by Observation~\ref{observacion:reduccion-maximo}.
    
    \item $\R \succ \S$: Since 
    $\R=\pair{\prefix{A}{p-i+1}\word{a}{p-i+2}{p}}{0}$,
     $\S =\pair{a_i\ldots a_p A_{i-1}}{0}$ and we assumed 
    $\word{a}{i}{p} = \prefix{A}{p-i+1}$,
        necessarily  
     $ \word{a}{p-i+2}{p}> \prefix{A}{i-1} $.  But this  contradicts that $\R$ is a  \maximal pair.
    
    \item $\S \succ \R$: 
There is a rotation of $\R$  which is $\succ$-greater than $\R$, contradicting that $\R$ is a \maximal pair.
\end{itemize}
We conclude that all suffixes  
$ \word{a}{i}{p}$ of  $R$ are  lexicographically smaller than 
$\prefix{A}{p-i}$.
Therefore, 
all suffixes of $\word{a}{i}{p-1}(a_{p}+1)$  of $Q$
are lexicographically smaller than or equal to~$\prefix{A}{p-i}$.
We already  argued that   $\Q$ is indeed  maximal. 
For any rotation of $\Q$ 
    of the form  \linebreak
    $\pair{\word{a}{i}{p-1}(a_p + 1)0^{pq}\word{a}{1}{i-1}}{0}$, 
    for $i>1$, we argued that 
    $\word{a}{i}{p-1}(a_p + 1)$ is lexicographically smaller than or equal to  $A_{p-i}$. 
    Moreover, $\word{a}{i}{p-1}(a_p + 1)0^{pq}$  is lexicographically smaller than 
    $Q=\prefix{A}{p-1}(a_p + 1)0^{pq}$.
    For any rotation of $\Q$ that starts with a suffix of $0^{pq}$ it can not be maximal, 
    because for any $m$,
    $\prefix{Q}{m} > 0^m$, otherwise $\A$  would not be  maximal.

We have the maximal  successive pairs
$\Q=\pair{Q}{0}$, $\theta \Q=\A=\pair{A}{0}$ and  $\B=\theta \A=\pair{B}{0}$.
From the arguments above,  $\preduction{Q} = Q$ and $\preduction{B}=B$.
Since  $\A=\pair{(A_p)^{q+1}}{0}$ and we assumed  $q>1$,
$A$ is not equal to $\preduction{A}$.
Then,
$
\wordconcat{\preduction{Q}}{\wordconcat{\preduction{A}}{\preduction{B}}} = Q\preduction{A}B$.
Observe that the last $pq$ symbols of $Q$ followed by  $\preduction{A}$, followed by the first $pq$ symbols of $B$  give rise to 
\[
\pair{0^{pq}(A_p)^{q+1}}{0},
\]
where $0$ is because 
 $p$ is a multiple of $k$.
We now identify in this pair  the $p-i$ rotations of $\A$ to the right.
Since  $\A=\pair{A}{0}$ where
$A=(A_i 0^{p-i})^{q+1}$,
 after  $r$ rotations to the right of $\A$, with $1 \leq r \leq p-i$,
we obtain   
\[
\pair{0^{r}(A_p)^q A_i0^{p-i-r}}{-r}.
\]

% Como obtuvimos las primeras $i$ rotaciones a izquierda y $p-i$ a derecha, obtuvimos todas las $p$ rotaciones.

{\em  Subcase  $q=0$.} 
We need to see that for the \maximal pairs
$\A=\pair{A}{0}$
where $A$ is reduced, that is $A=\preduction{A}$,
%\pair{(A_p)^{q+1}}{0}\] 
%with  $q = 0$ 
we can find all rotations of $\A$ in the constructed $X$.
If $A = 0^{k-1} 1 0^{n-k}$
then  $\suboperator{\A} = \B=\pair{B}{0}=\pair{ 0^n}{0}$, which is the last  \maximal pair in the list~$\M$ and $\preduction{B}=0^k$.
Then,
\[
\wordconcat{\preduction{A}}{\preduction{B}} =0^{k-1} 1 0^{n-k} 0^k = 0^{k-1} 1 0^{n},
\]
and we can find   $k$ left rotations of $\A$ which are of the form
\[
\pair{0^{k-1-r} 1 0^{n-k+r}}{r} \text{ for $r=0,1,\ldots  k-1$}.
\]
Now assume $A \neq 0^{k-1} 1 0^{n-k}$. Let's write  
\[
\A=\pair{A_i0^{n-i}}{0},
\]
with $i$ such that  $a_i > a_{i+1} = \ldots = a_{n} = 0$,
$\B=\pair{B}{0}=\suboperator{\A}$ and $\C=\pair{C}{0}=\suboperator{\B}=\suboperator{}^{2}{\A}$. Then $\B$ is of the form
\[
\B = \pair{\prefix{A}{i-1}(a_i - 1)(s-1)^{j-i}\word{b}{j+1}{p}}{0},
\]
with $j$ is the least multiple of  $k$ with  $j \geq i$. 
Since $A=\preduction{A}$, and by Lemma~\ref{lemma:vecinos} $B$ is a prefix of $\wordconcat{\preduction{B}}{\preduction{C}}$, then $AB$ is a prefix of $\wordconcat{\preduction{A}}{\wordconcat{\preduction{B}}{\preduction{C}}}$, and we can find the first $i$ left rotations of $\A$ in it, which are of the form 
\[
\pair{\word{a}{r+1}{n}\prefix{A}{r}}{r}, \text{ for $r=0,1,\ldots , i-1$}.
\]
It remains to find $n-i$ right rotations of $\A$, which are of the form 
\[
\pair{0^{r} \prefix{A}{i}0^{n-i-r}}{-r}
\]
for $r=1,\ldots,n-i$.
Equivalently we can write it as 
\[
\pair{0^{n-i-h} \prefix{A}{i}0^h}{(n-i-h)}
\]
for $h=0,1,\ldots,n-i-1$.

If $h = 0$ and $\prefix{A}{i} = (s-1)^{i}$ then the pair  is
$
\pair{0^{n-i}(s-1)^{i}}{n-i}$. 
We find it  in the  pair 
$\pair{0^n (s-1)^n}{0}$, 
which results 
from the concatenation of the words in the last  two  Lyndon  pairs in  
$\L$ and the first two, that is  
$\L_{M-1}$ and  $\L_{M}$, followed by 
 $\L_1$ and $\L_2$.

If $h \neq 0$ or $\prefix{A}{i} \neq (s-1)^i$ then we find the  right rotations  of $\A$ in the concatenation of the words of three Lyndon pairs, that we call $\wordconcat{\preduction{\Q_t }}{
 \wordconcat{\preduction{\P }}{\preduction{\R}}}$ and we define below.
 Recall that  $\A=\pair{A}{0}=\pair{A_i0^{n-i}}{0}$, with $i$ such that  $a_i > a_{i+1} = \ldots = a_{n} = 0$,  is a \maximal pair.
We claim there is a unique  pair $\Q$ of the form
\[
\Q = \pair{\prefix{A}{i} 0^h S}{0},
\]
where $|S|=n - (i+h)$  is not zero. 
Notice that $\Q$ may not be maximal.

Since $\Q$ is a  pair  $\succ$-greater than~$\A$,
we know $\Q$ is a predecessor of $\A$ by~$\suboperator{}$.  
%pero nunca menor dado que
%because $\A=\pair{\prefix{A}{j} 0^{n-j}}{0}$.
%If we start with  $\Q$ and we apply the  $\suboperator{}$ successively we obtain  $\A$.  
The pair $\Q_t$ is  the closest predecessor of $\Q$ by $\suboperator{}$ that is a maximal pair.
%Starting from the pair   $\Q=\pair{\prefix{A}{i} 0^h S}{0}$  
For this, consider the successive  application of the operator $\suboperator{}$,
which  allows us to traverse the list \linebreak
$(\theta^l \pair{(s-1)^n}{0})_{1\leq l\leq T}$.
Our interest is to traverse it backwards.
Here is  a diagram:
\medskip

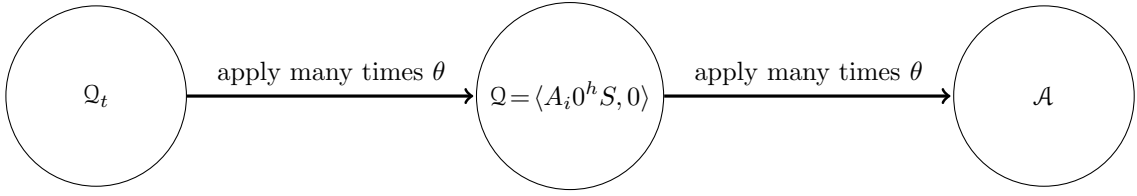
\begin{center}
\begin{tikzpicture}[shorten >=1pt,node distance=6.3cm,on grid,auto]
  \tikzstyle{every state}=[fill={rgb:black,1;white,10}]
    \node[circle,minimum size=1.5cm,draw]  (q_1)                    {\small $\Q_t$};
    \node[circle,minimum size=2cm,draw] (q_2)  [right of=q_1]    {\small $\Q\!=\!\pair{A_i0^hS}{0}$};
    \node[circle,minimum size=1.5cm,draw]  (q_3)  [right of=q_2]    {\small $\A$};
    \path[->]
    (q_1) edge[very thick]   node {\text{\small apply many times $\theta$}}    (q_2)
    (q_2) edge[very thick]   node {\text{\small apply many times $\theta$}}    (q_3);
   \end{tikzpicture}
\end{center}
\medskip

%Recall that  the list $\M$ contains just the \maximal pairs in this list.
For  $l=1,2,3, \ldots $ let  $\Q_l=\pair{Q_l}{0}$  be the pair  such that 
 $\suboperator{}^{l}{\Q_l}=\Q$.
%Entonces,  $\Q_1=\pair{Q_1}{0}$  el par tal que$\suboperator{\Q_1}=Q$
%Dado que  $h > 0$, existe un $a_{u_1}$ tal que 
%en la lista el  par anterior a $\pair{\prefix{A}{j}0^h B}{0}$ es
For $l=1$,
\begin{align*}
\Q_1&=\pair{Q_1}{0} = \pair{\prefix{A}{u_1-1} (a_{u_{1}} + 1) 0^{n-u_1}}{0},
\\
\suboperator{\Q_1} 
&=
\pair{(\prefix{A}{u_1}(s-1)^{r_1-u_1})^{w_1} \prefix{A}{v_1} S }{0}
= \pair{(\prefix{A}{r_1})^{w_1} \prefix{A}{v_1} S}{0} 
=\pair{ \prefix{A}{i} 0^h S}{0}
=\Q,
\end{align*}
where $r_1$ is the least multiple of $k$ such that $v_1 < r_1$, 
$a_{u_1} < (s-1)$,
$\prefix{A}{r_1} = \prefix{A}{u_1} (s-1)^{r_1 - u_1}$ and
$r_1-k \leq u_1 \leq r_1$.
If $\Q_1$ is a \maximal pair then fix $t=1$ and we have finished the search. Otherwise we consider the predecessor of  $\Q_1$ by $\suboperator{}$,
\[
\Q_2 = \pair{Q_2}{0}=\pair{\prefix{A}{u_2-1} (a_{u_2} + 1) 0^{n-u_2}}{0}
\]
such that
\[
\suboperator{\Q_2} =\pair{ (\prefix{A}{u_2}(s-1)^{r_2-u_2})^{w_2} \prefix{A}{v_2}}{0} = \pair{(\prefix{A}{r_2})^{w_2} \prefix{A}{v_2}}{0} =\Q_1,
\]
where  $r_2$ is the least multiple of  $k$, $v_2 < r_2$, and $a_{u_2} < (s-1)$ such that  $\prefix{A}{r_2} = \prefix{A}{u_2} (s-1)^{r_2 - u_2}$.
We know that 
$a_{u_2} < (s - 1)$ exists because
 $\Q_1 \neq \pair{(s-1)^n}{0}$,
 otherwise  $\Q_1$ would be a \maximal pair, and then there is $a_i < s - 1$ for  some $1 \leq i \leq r_2$.
 Notice that  $u_2 < u_1$.
If $\Q_2$ is a  \maximal pair then we fix  $t=2$ and we have finished the search. Otherwise, we repeat 
this procedure. In this way the predecessor of 
 $\Q_{l-1}$  by $\suboperator{}$ is 
\[
\Q_l = \pair{Q_l}{0} =\pair{\prefix{A}{u_l-1} (a_{u_l} + 1) 0^{n-u_l}}{0},
\]
such that 
\[
\suboperator{\Q_l} = \pair{(\prefix{A}{u_l}(s-1)^{r_l-u_l})^{w_l} \prefix{A}{v_l}}{0} = \pair{(\prefix{A}{r_l})^{w_l} \prefix{A}{v_l}}{0} = \Q_{l-1},
\]
where  $r_l$ is  multiple of $k$, $v_l < r_l$, $u_l < u_{l-1}$ and $a_{u_l} < (s-1)$ such that  
\[
\prefix{A}{r_l} = \prefix{A}{u_l} (s-1)^{r_l - u_l}.
\]
Eventually we find  $t$ such that $\Q_t$ is a \maximal pair. 
% Such a \maximal pair   $\Q_t$ exists  because in each step we have considered a pair $\succ$-greater than the previous one, and we know there is an initial pair   $\pair{(s-1)^n}{0}$.
%is the last one in the chain  que  es el par lexicográficamente máximo y es collar.
% \tropi{creo que no, igual puedo verlo rapido con el codigo que busca estos 3 Q P R}
% \tropi{
% ahi tengo esto para n = 6 y k = 2
% \\
% A 221110 Qr 221120 rP 221112 rR 221111
% \\f
% fijate que ahi si pegas Qr y rP tenes la rotacion (022111,1)
% pero A no es ninguno de esos 3 y Ai0h era 2211 (o sea la parte de Q que es Ai0h S) y tampoco esta
% }
% \vero{Entonces $Q$ no necesariamente es maximal!!!
% }
% mira aca tengo un ejemplo sobre Q no maximal
% tenes A 220020
% si tenes Ai0h = 22002, el Q va a ser 22002S, donde S es un solo simbolo
% despues te quedan estos 3 Q P R
%  Qt 220100 P 220021 R 220020
% y fijate que aplicando 1 vez el operador a Qt te da 220122, que es basicamente Ai0hS, con S = 2, porque te queda 220022
% entonces Q seria eso, 220022, pero no es maximal porque lo mejor es 222200
% asi que Q no sirve y tenes q aplicar el operador devuelta que te lleva a 220021, que es P, y bueno asi sigue
% dale! lo agrego despues de ir corrigiendo lo otro que dejaste en azul y repasando
% no pasa nada vero! mejor que quede todo bien jeje
% ahi dejo el chat borralo cuando ya no lo veas necesario jajaja (todo el texto que venimos hablando digo)
%
Consider now the three consecutive  \maximal pairs in the list $\M$,  
\[
\Q_t=\pair{Q_t}{0},  \P=\pair{P}{0}, \text{ and } \R=\pair{R}{0}.
\]
Let  $\overline{\Q_t}, \overline{\P}$ and $\overline{\R}$ be  the corresponding Lyndon pairs. $\P$ always exists, because it's either $\A$ or a maximal pair before it, and $\R$ is $\pair{0^n}{0}$ in the worst case.
Notice that  $\preduction{Q_t}$ ends with $0^{n-u_t}$
and, by  Lemma~\ref{lemma:vecinos},
 $\wordconcat{\preduction{P }}{\preduction{R}}$
starts with   $P$. Therefore, $\wordconcat{\preduction{\Q_t }}{
 \wordconcat{\preduction{\P }}{\preduction{\R}}}$ contains 
$0^{n-u_t} P = 0^{n-u_t} \prefix{A}{i} 0^h C$, for some~$C$.
Since   $\P$ is a \maximal pair $\succ$-greater than or equal to $\A$,  we can assert that its prefix is $\prefix{A}{i}0^h$.
Finally, we have
\[
u_t \leq u_1 \leq i + h,
\]
because
$
u_t < u_{t-1} < \ldots < u_1
$
and
$u_1 \leq i + h$
because $u_1$ was the 
position of a symbol in $A_i 0^h$, which has length $i+h$. 
Then,
\begin{align*}
u_t \leq u_1 \leq i + h 
& \iff
%- u_t \geq -u_1 \geq -i-h \\
%& \iff n - u_t \geq n - u_1 \geq n - i - h\\
%& \iff 
n - u_t \geq n - i - h.
\end{align*}
We conclude that
$\preduction{\Q_t}\preduction{\P}\preduction{\R}$ contains
\[
\pair{0^{n-i-h} \prefix{A}{i} 0^h}{(\wordlength{\preduction{Q_h}} - (n-i-h))}.
\]
This can be rewritten  as the pair
\[
 \pair{0^{n-i-h} \prefix{A}{i} 0^h}{-(n-i-h)}, 
\]
because $\wordlength{\preduction{Q_t}}$ is a multiple of $k$, hence $\Q_t$  has second component $0$.
\medskip

\subsubsection{\bf When $n|k$, $n\leq k$}
This proof is similar to the previous one, but now we need to use the corresponding definitions of $\suboperator{}$, the list  $\M$ and the notion of expansion  to define the list $\L$. The proof becomes simpler because it requires fewer cases, and each of these cases is simpler too.
The key observation for this proof is that each word $A$ of length $n$, repeated $k/n$ times,  determines a    Lyndon pair
$\pexpansion {A}=\pair{\pexpansion{A}}{0}$, which has exactly $|\pexpansion{A}|=k$ many different rotations.

\paragraph{\em Case $A = 0^n$.}
We can find the first $k-n+1$ left rotations of $A$ inside 
 $\pexpansion{\A} =\pair{\pexpansion{A}}{0}=\pair{0^k}{0}$, which are of the form $\pair{0^n}{i}$, with $0\leq i < k-n+1$.
Consider the \maximal pairs for $0 \leq i < n-1$,
\[
\B = \pair{B}{0}=\pair{0^i 1 0^{n-1-i}}{0}.
\]
For each of these $\B$,
%the next  \maximal pair  is
\[
\suboperator{\B} = \C = \pair{C}{0}= \pair{0^{i+1} (s-1)^{n-1-i}}{0}.
\]
Since $\pexpansion{\B}$ and $\pexpansion{\C}$ are consecutive Lyndon pairs in $\L$, the construction of $X$ puts $\pexpansion{B}$   followed by $\pexpansion{C}$. 
Notice that for each $i=0,1,\ldots , k-1$,
\[
\wordconcat{\pexpansion{\B}}{\pexpansion{\C}} = \pair{(0^i 1 0^{n-1-i})^{k/n - 1} 0^i 1 0^n (s-1)^{n-1-i} (0^{i+1} (s-1)^{n-1-i})^{k/n - 1}}{0} 
\]
gives rise to the pair $\pair{0^n}{i+1}$. 
Thus, we have all  the rotations  of the pair $\A = \pair{0^n}{0}$, which are   $\pair{0^n}{0}$, $\ldots $,  $\pair{0^n}{k-1}$.

\paragraph{\em Case $A > 0^n$.}
Let $A=\word{a}{1}{n}$, and let $j$ be the largest such that $a_j>0$. Fix $j$.
Since the pair $\A$ is different from $\pair{0^n}{0}$, it has necessarily a \maximal pair successor in the list $\M$,
\[
 \B = \suboperator{\A} = \pair{\prefix{A}{j-1}(a_j - 1)(s-1)^{n-j}}{0},
\]
Since $\pexpansion{\B} = \B^{k/n}$, 
\[
\wordconcat{\pexpansion{\A}}{\pexpansion{\B}} =  \pair{A^{k/n}(\prefix{A}{j-1}(a_j - 1)(s-1)^{n-j})^{k/n}}{0},
\]
and it also yields the first $k - n + j$ left rotations of $\A$, which are 
$\pair{\word{a}{(i\mod n) + 1}{n}\prefix{A}{i \mod n}}{i}$, for $0 \leq i < k - n + j$.
Note that all of the rotations for $A = (s-1)^n$ are contained inside these left rotations, as $j = n$, and then $0 \leq i < k$.

If $A = (s-1)^n$ then  the $k$ rotations of $\pexpansion{\A}=\pair{\pexpansion{A}}{0}$  are considered above by taking  $j = n$ and  $0 \leq i < k$.
In case $A\not=(s-1)^n$ we have considered just $k-n-j$ rotations 
% It remains to find $n-j$ right rotations of $\A$, 
which are 
%of the form 
% \[
% \pair{0^{r} \prefix{A}{j}0^{n-j-r}}{-r}
% \]
% for $r=1,\ldots,n-j$.
% Equivalently we can write it as 
\[
\pair{0^{n-j-h} \prefix{A}{j}0^h}{-(n-j-h)}
\]
for $h=0,1,\ldots,n-j-1$.

If $h = 0$ and $\prefix{A}{j} = (s-1)^{j}$ then the pair  is
$
\pair{0^{n-j}(s-1)^{j}}{j-n}$. 
We find it in the  pair 
$\pair{0^k (s-1)^k}{0}$, 
which results 
from the concatenation of the last and first Lyndon pairs in  
$\L$, which are  
$\L_{M-1}$ and $\L_{1}$ respectively.

If $h \neq 0$ or $\prefix{A}{j} \neq (s-1)^j$ then
the pairs are $\pair{0^{n-j-h} \prefix{A}{j}0^h}{-(n-j-h)}
$ for $h=1, \ldots n-j-1$, and we find them in 
the concatenation of the words of two Lyndon pairs, that we call $\wordconcat{\pexpansion{\Q }}{\pexpansion{\P }}$ and we define below.
Let $i$ be the largest such that  $a_i < s-1$ and $i \leq j+h$, and let $\Q=\pair{Q}{0}=\pair{\prefix{A}{i-1} (a_i + 1) 0^{n-i}}{0}$.
By Observation~\ref{obs:todos-maximales},
$\Q$ is a Lyndon pair, different from~$\pair{0^n}{0}$. Notice  that $\Q$ is well defined because
there is some $i$, with $1\leq i\leq n$ such that $1\leq a_i + 1\leq s-1$. 
%is well defined and $i > 0$.
To see this  notice that  
either
$h = 0$ but $\prefix{A}{j}\neq (s-1)^j$, hence,  there exists  $a_i < (s-1)$; or $h > 0$ and then $i = j+h$.
%either $h > 0$ and then $i = j+h$; or $h = 0$ but $\prefix{A}{j} \neq (s-1)^j$, hence,  there exists  $a_i < (s-1)$.
%
Let $\P$ be the successor of $\Q$ in $\M$, which is of the form
\[
\pair{\prefix{A}{i}(s-1)^{n-i}}{0}.
\]
% As these are consecutive in $\M$, 
The construction concatenates $\wordconcat{\pexpansion{\Q }}{\pexpansion{\P }}$, resulting in
\[
\pair{\prefix{A}{i-1} (a_i + 1) 0^{n-i}}{0}^{k/n} \pair{\prefix{A}{i} (s-1)^{n-i}}{0}^{k/n}.
\]
Finally, if we take the suffix $\prefix{\pexpansion{\Q}}{n-j-h}$ followed by the prefix $\prefix{\pexpansion{\P}}{n-j-h}$ it results in the pair $\pair{0^{n-j-h}\prefix{A}{j+h}}{-(n-j-h)}$.
Given that we can do this for every possible $h=0, 1,\ldots n-j-1$, we have obtained the wanted $n-j$ right rotations. Consequently, we have found all the rotations of $\A$.

\color{black}

\subsection{Proof that necklace $X$ is the lexicographically greatest $(n,k)$-perfect necklace}
\label{sec:3}

The necklace $X$ is the concatenation of all the words in the 
Lyndon pairs in the list \linebreak
$\L=(\pair{w_i}{0})_{1\leq i\leq M}$.
Let $\ell_0=0$ and let $(\ell_i)_{1\leq i\leq M}$
such that $\ell_i=\sum_{m=1}^{i}|w_m|$.
We already showed $|X|=s^n k$, hence,  $\ell_M= s^n k$.
The necklace 
$X= w_1 w_2 \ldots w_M = a_0 a_1 , \ldots a_{s^n k-1}$,
where 
for $i=1, 2, \ldots,  M$,  $ w_i=a_{\ell_{i-1}},\ldots a_{\ell_{i}-1}$.
Since each $w_i$  is a Lyndon word, $w_i$ is lexicographically greater than all of its rotations.
%and $w_i$ is lexicographically greater than all the rotations of all the $w_m$, for $m>i$.
Now, instead of looking at an index $i$ from $1$ to $M$, 
let's consider an index $j$ running through all the positions of $X$, from $0$ to $s^nk-1$.
Consider the set of pairs 
\[
{\mathcal S}=\{ \pair{a_j\ldots a_{j+n-1}}{j\mod k} : 0\leq j\leq s^nk-1 \}.
\]
Since we already proved that $X$ is  $(n,k)$-perfect, the set ${\mathcal S}$
is exactly $\Sigma^n\times\Z_k$.
The construction guarantees that 
 each position $j=0, \ldots, s^nk-1$  is the start  of one of the different pairs in~${\mathcal S}$.
 As a consequence of the order of the elements in  the list $\L$,
when   $j=\ell_i-1$  for some~$i$, $1\leq i\leq |M|$, the pair
$\pair{a_j\ldots  a_{j+n-1}}{0}$
is the $\succ$-greatest among the  remaining pairs with second component~$0$.
If  $k|n$ and $k<n$, each $a_j$  is a symbol in   $n$ different pairs in $\Sigma^n\times \Z_k$, $n/k$ of them with second component $0$.
For example, 
$a_{j+n}$ is in $n$ pairs,
\[
\pair{a_{j+1}\ldots a_{j+n} }{(j+1)\mod k}, \ldots,
\pair{a_{j+n}\ldots a_{j+2n-1}}{(j+n)\mod k}.
\]
If  $n|k$ and $n\leq k$, each $a_j$  is a symbol in  $k$ different pairs in $\Sigma^k\times \Z_k$.
There is exactly one of these with second component $0$ and it has the form  $\pair{A^{k/n}}{0}$, with $A\in\Sigma^n$. 
For example, 
$a_{j+k}$ is in $k$ pairs,
\[
\pair{a_{j+1}\ldots a_{j+k} }{(j+1)\mod k}, \ldots,
\pair{a_{j+k}\ldots a_{j+2k-1}}{(j+k)\mod k}.
\]
Implicitly our construction defines a function  $f:\Sigma^n\to\Sigma$, such that for each $j=0, \ldots, s^nk-1$,
\[
a_{j+n} = f(a_{j}, \ldots , a_{j+n-1}).
\]
%The construction  guarantees that the  pairs $\pair{A}{0}$ in $\Sigma^n\times\Z_k$ occur in $X$ from left to right 
%in decreasing $\succ$-order, which means decreasing lexicographical order of the first component because the second component is $0$.
% The construction  starts with the pair $\pair{(s-1)^n}{0}$ and then, one after the other, it  puts all the  remaining pairs $\pair{A}{0}$. 
Suppose
 $b_0, b_1, \ldots, b_{s^n k-1}$  is another  $(n,k)$-perfect necklace and there is a position $p$ such that 
 $b_0=a_0, b_1=a_1, \ldots, b_{p-1}=a_{p-1}$  but $ b_{p}> a_{p}$.
Then, there is a position $q$ congruent to $0$ modulo $k$ such that
%$q\equiv k\mod 0$,
$p-\max(n,k)+1\leq q\leq p+\max(n,k)-1$,
 and 
\[
\word{b}{q}{q+\max(n,k)-1}>
\word{a}{q}{q+\max(n,k)-1}.
\]
Then, the respective  pairs starting at position $q$ satisfy,
\[
\pair{\word{b}{q}{q+n-1}}{0} \succ \pair{\word{a}{q}{q+n-1}}{0}.
\]
But this contradicts that our construction puts in $a_0 \ldots a_{s^nk-1}$ the pairs $\pair{A}{0}$ in $\Sigma^n\times\Z_k$ in $\succ$-decreasing ordering.
This completes the proof of Theorem~\ref{thm:perfectp}.

\section*{Acknowledgements}
We thank the anonymous reviewer for their careful reading of the manuscript and for the comments and suggestions, which significantly improved the presentation.
This work  is supported by the projects
CONICET PIP 11220210100220CO and
UBACyT 20020220100065BA.
\medskip

\bibliographystyle{plain}
\bibliography{lyndon}
\end{document}